\documentclass[12pt]{article}

\usepackage{latexsym,amsmath,amssymb,epic}
\usepackage{latexsym,amssymb,epic}
\usepackage{amsmath,amsthm}
\usepackage{color}

\usepackage{latexsym,amssymb,epic}
\usepackage{amsmath,amsthm}
\usepackage{color}

\usepackage{hyperref}
\hypersetup{
	colorlinks=true, %set true if you want colored links
	linktoc=all, %set to all if you want both sections and subsections linked
	linkcolor=blue} %choose some color if you want links to stand out

\newtheorem{theorem}{Theorem}[section]
\newtheorem{conjecture}{Conjecture}[section]
\newtheorem{lemma}{Lemma}[section]

\newtheorem{corollary}{Corollary}[section]

\begin{document}

\title{Congruences for the coefficients of the Gordon and McIntosh mock theta function $\xi(q)$}

\author{Robson da Silva and James A. Sellers}

\date{}
\maketitle

\begin{abstract}
Recently Gordon and McIntosh introduced the third order mock theta function $\xi(q)$ defined by 
$$
\xi(q)=1+2\sum_{n=1}^{\infty}\frac{q^{6n^2-6n+1}}{(q;q^6)_{n}(q^5;q^6)_{n}}.
$$
Our goal in this paper is to study arithmetic properties of the coefficients of this function. We present a number of such properties, including several infinite families of Ramanujan--like congruences.
\end{abstract}



\section{Introduction}

In his last letter to Hardy in 1920, Ramanujan introduced the notion of a mock theta function. He listed 17 such functions having orders 3, 5, and 7. Since then, other mock theta functions have been found. Gordon and McIntosh \cite{G-M}, for example, introduced many additional such functions, including the following of order 3:
\begin{equation}
\xi(q)=1+2\sum_{n=1}^{\infty}\frac{q^{6n^2-6n+1}}{(q;q^6)_{n}(q^5;q^6)_{n}},
\label{xi}
\end{equation}
where we use the standard $q$-series notation:
\begin{align*}
(a;q)_0 & = 1, \\
(a;q)_n & = (1-a)(1-aq) \cdots (1-aq^{n-1}), \forall n \geq 1, \\ 
(a;q)_{\infty} & = \lim_{n \to \infty} (a;q)_n, |q|<1.
%(a_1, a_2, \ldots, a_k;q)_{\infty} = (a_1;q)_{\infty}(a_2;q)_{\infty} \cdots (a_k;q)_{\infty}.
\end{align*}

Arithmetic properties of the coefficients of mock theta functions have received a great deal of attention. For instance, Zhang and Shi \cite{Z-S} recently proved seven congruences satisfied by the coefficients of the mock theta function $\beta(q)$ introduced by McIntosh. In a recent paper, Brietzke, da Silva, and Sellers \cite{BSS} found a number of arithmetic properties satisfied by the coefficients of the mock theta function $V_0(q)$, introduced by Gordon and McIntosh \cite{G-M-0}. Andrews et al. \cite{APSY} prove a number of congruences for the partition functions $p_{\omega}(n)$ and $p_{\nu}(n)$, introduced in \cite{Andrews1}, associated with the third order mock theta functions $\omega(q)$ and $\nu(q)$, where $\omega(q)$ is defined below and 
$$\nu(q) = \sum_{n=0}^{\infty} \frac{q^{n(n+1)}}{(-q;q^2)_{n+1}}.$$
In a subsequent paper, Wang \cite{Wang} presented some additional congruences for both $p_{\omega}(n)$ and $p_{\nu}(n)$.

This paper is devoted to exploring arithmetic properties of the coefficients $p_{\xi}(n)$ defined by 
\begin{equation}
\sum_{n=0}^{\infty} p_{\xi}(n)q^n = \xi(q).
\label{Eq4}
\end{equation} 
It is clear from \eqref{xi} that $p_{\xi}(n)$ is even for all $n \geq 1$. In Sections \ref{Main} and \ref{sec:addl_congs}, we present other arithmetic properties of $p_{\xi}(n)$, including some infinite families of congruences.

\section{Preliminaries}
\label{sec:preliminaries}

McIntosh \cite[Theorem 3]{McIntosh} proved a number of mock theta conjectures, including 
\begin{align}
\omega(q) & = g_3(q,q^2) \textrm{\ \ \ and}  \label{Eq1} \\
\xi(q) & = q^2g_3(q^3,q^6) + \frac{(q^2;q^2)_{\infty}^{4}}{(q;q)_{\infty}^2(q^6;q^6)_{\infty}}, \label{Eq2}
\end{align}
where
$$g_3(a,q) = \sum_{n=0}^{\infty} \frac{(-q;q)_n q^{n(n+1)/2}}{(a;q)_{n+1}(a^{-1}q;q)_{n+1}}$$
and $\omega(q)$ is the third order mock theta functions given by
$$\omega(q) = \sum_{n=0}^{\infty} \frac{q^{2n(n+1)}}{(q;q^2)_{n+1}^2}.$$
It follows from \eqref{xi}, \eqref{Eq1}, and \eqref{Eq2} that
\begin{equation}
\xi(q) = q^2\omega(q^3) + \frac{(q^2;q^2)_{\infty}^{4}}{(q;q)_{\infty}^2(q^6;q^6)_{\infty}}.
\label{Eq3}
\end{equation}

Throughout the remainder of this paper, we define 
$$f_k := (q^k;q^k)_{\infty}$$
in order to shorten the notation. Combining \eqref{Eq3} and \eqref{Eq4}, we have
\begin{equation}
\sum_{n=0}^{\infty} p_{\xi}(n)q^n = q^2\omega(q^3) + \frac{f_2^{4}}{f_1^2f_6}.
\label{Eq5}
\end{equation}

We recall Ramanujan's theta functions
\begin{align}
f(a,b) & :=\sum_{n=-\infty}^\infty a^\frac{n(n+1)}{2}b^\frac{n(n-1)}{2}, \mbox{ for } |ab|<1, \nonumber \\
\phi(q) & := f(q,q) = \sum_{n=-\infty}^{\infty} q^{n^2} = \frac{f_2^5}{f_1^2f_4^2}, \textrm{\ \ and} \label{Rama1} \\
\psi(q) & := f(q,q^3) = \sum_{n=0}^{\infty} q^{n(n+1)/2} = \frac{f_2^2}{f_1}. \label{Rama2}
\end{align}

The function $\phi(q)$ satisfies many identities, including (see \cite[(22.4)]{Berndt})
\begin{equation}
{\phi(-q)} = \frac{f_1^2}{f_2}. \label{Eq18}
\end{equation}

In some of the proofs, we employ the classical Jacobi's identity (see \cite[Theorem 1.3.9]{Berndt1})
\begin{equation}
f_1^3 = \sum_{n=0}^{\infty} (-1)^n (2n+1) q^{n(n+1)/2}.
\label{Jacobi}
\end{equation}

We note the following identities which will be used below.  

\begin{lemma} The following 2-dissection identities hold.
\begin{eqnarray}
\frac{1}{f_{1}^{2}} & = & \frac{f_{8}^{5}}{f_{2}^{5}f_{16}^{2}} + 2q\frac{f_{4}^{2}f_{16}^{2}}{f_{2}^{5}f_{8}}, \label{Eq17} \\
{f_{1}^{2}} & = & \frac{f_{2}f_8^5}{f_{4}^{2}f_{16}^{2}} -2q\frac{f_{2}f_{16}^{2}}{f_{8}}, \label{Eq31} \\
\frac{1}{f_{1}^{4}} & = & \frac{f_{4}^{14}}{f_{2}^{14}f_{8}^{4}} + 4q\frac{f_{4}^{2}f_{8}^{4}}{f_{2}^{10}}, \label{Eq29} \\
\displaystyle\frac{f_{3}}{f_{1}} & = & \displaystyle\frac{f_{4}f_{6}f_{16}f_{24}^{2}}{f_{2}^{2}f_{8}f_{12}f_{48}} + q\displaystyle\frac{f_{6}f_{8}^{2}f_{48}}{f_{2}^{2}f_{16}f_{24}}, \label{Eq8} \\
\displaystyle\frac{f_{3}^2}{f_{1}^2} & = & \displaystyle\frac{f_{4}^4f_{6}f_{12}^{2}}{f_{2}^{5}f_{8}f_{24}} + 2q\displaystyle\frac{f_{4}f_{6}^{2}f_{8}f_{24}}{f_{2}^{4}f_{12}}, \label{Eq9} \\
\displaystyle\frac{f_{1}^3}{f_{3}} & = & \displaystyle\frac{f_{4}^3}{f_{12}} - 3q\displaystyle\frac{f_{2}^{2}f_{12}^3}{f_{4}f_{6}^2} \label{Eq45} \\
\displaystyle\frac{f_{3}}{f_{1}^3} & = & \displaystyle\frac{f_{4}^6f_{6}^3}{f_{2}^{9}f_{12}^2} + 3q\displaystyle\frac{f_{4}^2f_{6}f_{12}^2}{f_{2}^{7}}, \label{Eq11} \\
\displaystyle\frac{1}{f_{1}f_{3}} & = & \displaystyle\frac{f_{8}^{2}f_{12}^{5}}{f_{2}^{2}f_{4}f_{6}^{4}f_{24}^{2}} + q\displaystyle\frac{f_{4}^{5}f_{24}^{2}}{f_{2}^{4}f_{6}^{2}f_{8}^{2}f_{12}} \label{Eq38}
\end{eqnarray}
\end{lemma}

\begin{proof}
By Entry 25 (i), (ii), (v), and (vi) in \cite[p. 40]{Berndt}, we have
\begin{align}
\phi(q) & = \phi(q^4) + 2q\psi(q^8),
\label{phi-2-dissec} \\
\phi(q)^2 & = \phi(q^2)^2 + 4q\psi(q^4)^2.
\label{phi^2-2-dissec}
\end{align} 
Using \eqref{Rama1} and \eqref{Rama2} we can rewrite \eqref{phi-2-dissec} in the form
\begin{equation*}
\frac{f_{2}^{5}}{f_{1}^{2}f_{4}^{2}} = \frac{f_{8}^{5}}{f_{4}^{2}f_{16}^{2}} + 2q\frac{f_{16}^{2}}{f_{8}},
\end{equation*}
from which we obtain \eqref{Eq17} after multiplying both sides by $\frac{f_4^2}{f_2^5}$. Identity \eqref{Eq31} can be easily deduced from \eqref{Eq17} using the procedure described in Section 30.10 of \cite{H}.

By \eqref{Rama1} and \eqref{Rama2} we can rewrite \eqref{Eq1} in the form
\begin{equation*}
\frac{f_{2}^{10}}{f_{1}^{4}f_{4}^{4}} = \frac{f_{4}^{10}}{f_{2}^{4}f_{8}^{4}} + 4q\frac{f_{8}^{4}}{f_{4}^{2}},
\end{equation*}
from which we obtain \eqref{Eq29}.

Identities \eqref{Eq8}, \eqref{Eq9}, and \eqref{Eq38} are equations (30.10.3), (30.9.9), and (30.12.3) of \cite{H}, respectively. Finally, for proofs of \eqref{Eq45} and \eqref{Eq11} see \cite[Lemma 4]{Naika}.
\end{proof}

The next lemma exhibits the 3-dissections of $\psi(q)$ and $1/\phi(-q)$.

\begin{lemma} We have
\begin{eqnarray}
\psi(q) & = & \displaystyle\frac{f_{6}f_{9}^{2}}{f_{3}f_{18}} + q\displaystyle\frac{f_{18}^{2}}{f_{9}}, \label{3_dissec_psi} \\
\displaystyle\frac{1}{\phi(-q)} & = & \displaystyle\frac{f_{6}^4f_{9}^6}{f_{3}^8f_{18}^3} + 2q\displaystyle\frac{f_{6}^3f_{9}^3}{f_{3}^7} + 4q^2\displaystyle\frac{f_{6}^2f_{18}^3}{f_{3}^6}. \label{3_dissec_1/phi}
%\phi(q) & = & \displaystyle\frac{f_{18}^{5}}{f_{9}^2f_{36}^2} + 2q\displaystyle\frac{f_{6}^2f_{9}f_{36}}{f_{3}f_{12}f_{18}}, \label{3_dissec_phi} \\
\end{eqnarray}
\label{lemma2}
\end{lemma}

\begin{proof}
Identity \eqref{3_dissec_psi} is equation (14.3.3) of \cite{H}. A proof of \eqref{3_dissec_1/phi} can be seen in \cite{H-S-1}.
\end{proof}

\section{Dissections for $p_{\xi}(n)$}
\label{sec:ids}

This section is devoted to proving the 2-, 3-, and 4-dissections of \eqref{Eq4}. We begin with the 2-dissection.

\begin{theorem} We have
\begin{align}
2\sum_{n=0}^{\infty} p_{\xi}(2n+1)q^{n+1} & = \frac{f_{6}^6f_{12}}{f_3^4f_{24}^2} - f(q^{12})  +4q\frac{f_2^2f_{8}^2}{f_1f_3f_4}, \textrm{\ \ and} \label{Eq36} \\
\sum_{n=0}^{\infty} p_{\xi}(2n)q^{n} & = q\frac{f_{6}^8f_{24}^2}{f_3^4f_{12}^5} - q^4\omega(-q^{6}) + \frac{f_4^5}{f_1f_3f_{8}^2}. \label{Eq37}
\end{align}
\end{theorem}

\begin{proof}
We start with equation (4) of \cite{APSY}:
\begin{equation*}
f(q^8) +2q\omega(q) + 2q^3\omega(-q^4) = F(q),
\end{equation*}
where $f(q)$ is the mock theta function 
$$f(q) = \sum_{n=0}^{\infty} \frac{q^{n^2}}{(-q;q)_{n}^{2}}$$
and 
\begin{equation*}
F(q) = \frac{\phi(q)\phi(q^2)^2}{f_4^2} = \frac{f_2f_{6}^6}{f_1^2f_{12}^4}.
\end{equation*}
Thus,
\begin{equation*}
f(q^{24}) +2q^3\omega(q^3) + 2q^9\omega(-q^{12}) = F(q^3).
\end{equation*}
Using \eqref{Eq3}, it follows that
\begin{equation}
2\sum_{n=0}^{\infty} p_{\xi}(n)q^{n+1} = F(q^3) - f(q^{24}) - 2q^9\omega(-q^{12}) + 2q\frac{f_2^4}{f_1^2f_6}.
\label{Eq33}
\end{equation}
By \eqref{Eq17}, we have
\begin{equation*}
F(q^3) = \frac{f_{12}^6f_{24}}{f_6^4f_{48}^2} + 2q^3\frac{f_{12}^8f_{48}^2}{f_6^4f_{24}^5},
\end{equation*}
which along with \eqref{Eq17} allows us to rewrite \eqref{Eq33} as
\begin{align*}
2\sum_{n=0}^{\infty} p_{\xi}(n)q^{n+1} & = \frac{f_{12}^6f_{24}}{f_6^4f_{48}^2} + 2q^3\frac{f_{12}^8f_{48}^2}{f_6^4f_{24}^5} - f(q^{24}) - 2q^9\omega(-q^{12}) \\ 
& \ \ \ \ + 2q\frac{f_8^5}{f_2f_6f_{16}^2} +4q^2\frac{f_4^2f_{16}^2}{f_2f_6f_8}.
\end{align*}
Thus,
\begin{align}
2\sum_{n=0}^{\infty} p_{\xi}(2n+1)q^{2n+2} & = \frac{f_{12}^6f_{24}}{f_6^4f_{48}^2} - f(q^{24})  +4q^2\frac{f_4^2f_{16}^2}{f_2f_6f_8}, \textrm{\ \ and} \label{Eq34} \\
\sum_{n=0}^{\infty} p_{\xi}(2n)q^{2n+1} & = q^3\frac{f_{12}^8f_{48}^2}{f_6^4f_{24}^5} - q^9\omega(-q^{12}) + q\frac{f_8^5}{f_2f_6f_{16}^2}. \label{Eq35}
\end{align}
Dividing \eqref{Eq35} by $q$ and replacing $q^2$ by $q$ in the resulting identity and in \eqref{Eq34}, we obtain \eqref{Eq36} and \eqref{Eq37}.
\end{proof}

The next theorem exhibits the 3-dissection of \eqref{Eq4}. 

\begin{theorem} We have
\begin{align}
\sum_{n=0}^{\infty} p_{\xi}(3n)q^n & = \frac{f_{2}f_{3}^4}{f_{1}^2f_{6}^2}, \label{xi_3n} \\
\sum_{n=0}^{\infty} p_{\xi}(3n+1)q^n & = 2\frac{f_{3}f_{6}}{f_{1}}, \textrm{\ \ and} \label{xi_3n+1} \\
\sum_{n=0}^{\infty} p_{\xi}(3n+2)q^n & =  \omega(q) + \frac{f_{6}^4}{f_2f_{3}^2}. \label{xi_3n+2}
\end{align}
\label{T0}
\end{theorem}

\begin{proof}
In view of \eqref{Rama2}, we rewrite \eqref{Eq5} as
\begin{equation*}
\sum_{n=0}^{\infty} p_{\xi}(n)q^n = q^2\omega(q^3) + \frac{\psi(q)^2}{f_6}.
\end{equation*}
Using \eqref{3_dissec_psi}, we obtain 
\begin{equation}
\sum_{n=0}^{\infty} p_{\xi}(n)q^n = q^2\omega(q^3) + \frac{f_{6}f_{9}^4}{f_{3}^2f_{18}^2} +2q\frac{f_{9}f_{18}}{f_{3}} + q^2\frac{f_{18}^4}{f_6f_{9}^2}.
\label{3_dissec_xi}
\end{equation}	
Extracting the terms of the form $q^{3n+r}$ on both sides of \eqref{3_dissec_xi}, for $r \in \{0,1,2\}$, dividing both sides of the resulting identity by $q^r$ and then replacing $q^3$ by $q$, we obtain the desired results.
\end{proof}

We close this section with the 4-dissection of \eqref{Eq4}. 

\begin{theorem} We have
	\begin{align}
	\sum_{n=0}^{\infty} p_{\xi}(4n)q^n & = 4q^2\frac{f_{12}^6}{f_{3}^2f_{6}^3} - q^2 \omega(-q^3) + \frac{f_2^4f_6^5}{f_1^2f_3^4f_{12}^2}, \label{xi_4n} \\
	\sum_{n=0}^{\infty} p_{\xi}(4n+1)q^n & = 2q\frac{f_{6}^3f_{12}^2}{f_{3}^4} + 2\frac{f_{4}^4f_{6}^5}{f_{2}^2f_{3}^4f_{12}^2}, \label{xi_4n+1} \\
	\sum_{n=0}^{\infty} p_{\xi}(4n+2)q^n & = \frac{f_{6}^9}{f_{3}^6f_{12}^2} +\frac{f_{2}^{10}f_{12}^2}{f_{1}^4f_{3}^2f_{4}^4f_{6}}, \textrm{\ \ and} \label{xi_4n+2} \\
	2\sum_{n=0}^{\infty} p_{\xi}(4n+3)q^{n+1} & = \frac{f_{6}^{15}}{f_3^8f_{12}^6} - f(q^6) + 4q\frac{f_{2}^4f_{12}^2}{f_{1}^2f_{3}^2f_{6}}. \label{xi_4n+3}
	\end{align}
	\label{T6}
\end{theorem}

\begin{proof} 
In order to prove \eqref{xi_4n}, we use \eqref{Eq17} and \eqref{Eq38} to obtain the even part of \eqref{Eq37}, which is given by 
\begin{align*}
\sum_{n=0}^{\infty} p_{\xi}(4n)q^{2n} & = 4q^4\frac{f_{24}^6}{f_6^2f_{12}^3} -q^4\omega(-q^6) + \frac{f_4^{4}f_{12}^5}{f_2^2f_6^4f_{24}^2}.
\end{align*}
Replacing $q^2$ by $q$ we obtain \eqref{xi_4n}.	
	
Using \eqref{Eq17} and \eqref{Eq38} we can extract the odd part of \eqref{Eq36}:
\begin{align*}
2\sum_{n=0}^{\infty} p_{\xi}(4n+1)q^{2n+1} & = 4q^3\frac{f_{12}^3f_{24}^2}{f_6^4}  + 4q\frac{f_8^4f_{12}^5}{f_4^2f_6^4f_{24}^2}.
\end{align*}
After simplifications we arrive at \eqref{xi_4n+1}.

Next, extracting the odd part of \eqref{Eq37} with the help of \eqref{Eq17} and \eqref{Eq38} yields 
\begin{align*}
\sum_{n=0}^{\infty} p_{\xi}(4n+2)q^{2n+1} & = q\frac{f_{12}^9}{f_6^6f_{24}^2}  + q\frac{f_4^{10}f_{24}^2}{f_2^4f_6^2f_{8}^4f_{12}},
\end{align*}
which, after simplifications, gives us \eqref{xi_4n+2}.

In order to obtain \eqref{xi_4n+3}, we use \eqref{Eq17} and \eqref{Eq38} in \eqref{Eq36} to extract its even part:
\begin{align*}
2\sum_{n=0}^{\infty} p_{\xi}(4n+3)q^{2n+2} & = \frac{f_{12}^{15}}{f_6^8f_{24}^6}  -f(q^{12}) + 4q^2\frac{f_4^4f_{24}^2}{f_2^2f_6^2f_{12}}.
\end{align*}
Replacing $q^2$ by $q$ in this identity, we obtain \eqref{xi_4n+3}.
%\begin{align}
%2\sum_{n=0}^{\infty} p_{\xi}(4n+3)q^{n+1} & = \frac{f_{6}^{15}}{f_3^8f_{12}^6}  -f(q^{6}) + %4q\frac{f_2^4f_{12}^2}{f_1^2f_3^2f_{6}}. \label{Eq40}
%\end{align}
%Now we recall the following identity involving $f(q)$ and $\phi(q)$ (see \cite[eq. (6.5)]{Andrews1})
%\begin{equation*}
%f(q) = 2\phi(-q) - \frac{(q;q)_{\infty}}{(-q;q)_{\infty}^2} = 2\phi(-q) - \frac{f_1^3}{f_2^2}.
%\end{equation*}
%Therefore \eqref{xi_4n+3} follows by using this identity in \eqref{Eq40}.
\end{proof}

\section{Arithmetic properties of $p_{\xi}(n)$}
\label{Main}
Our first observation provides a characterization of $p_{\xi}(3n) \pmod{4}.$

\begin{theorem} For all $n \geq 0$, we have
$$p_{\xi}(3n) \equiv  
\begin{cases}
1 \pmod{4}, & \mbox{if $n=0$},\\
2 \pmod{4}, & \mbox{if $n$ is a square}, \\
0 \pmod{4}, & \mbox{otherwise}.
\end{cases}$$	
\label{T2}
\end{theorem}

\begin{proof} By \eqref{xi_3n}, using \eqref{Eq18} and the fact that $f_k^4 \equiv f_{2k}^2 \pmod{4}$ for all $k\geq 1,$ it follows that
\begin{align*}
\sum_{n=0}^{\infty} p_{\xi}(3n)q^n & = \frac{f_{2}f_{3}^4}{f_{1}^2f_{6}^2} \equiv \frac{f_2}{f_1^2} = \frac{f_1^2f_2}{f_1^4} \equiv \frac{f_1^2}{f_ 2} = \phi(-q) \pmod{4}.
\end{align*}
By \eqref{Rama1}, we obtain
\begin{align*}
\sum_{n=0}^{\infty} p_{\xi}(3n)q^n & \equiv \sum_{n=-\infty}^{\infty} (-1)^n q^{n^2} \equiv 1 + 2\sum_{n=1}^{\infty} q^{n^2} \pmod{4},
\end{align*}
which completes the proof.
\end{proof}

Theorem \ref{T2} yields an infinite family of Ramanujan--like congruences modulo $4$.

\begin{corollary} For all primes $p > 3$ and all $n\geq 0,$ we have
$$p_{\xi}(3(pn+r)) \equiv 0 \pmod{4},$$
if $3r$ is a quadratic nonresidue modulo $p$.
\end{corollary}

\begin{proof}
If $3(pn+r) = k^2$, then $3r \equiv k^2 \pmod{p}$, which contradicts the fact that $3r$ is a quadratic nonresidue modulo $p$.
\end{proof}

Since $\gcd(3, p) = 1$, among the $p-1$ residues modulo $p$, we have $\frac{p-1}{2}$ residues $r$ for which $3r$ is a quadratic nonresidue modulo $p$. Thus, for instance, the above corollary yields the following congruences:
\begin{align*}
p_{\xi}(15n+k) & \equiv 0 \pmod{4}, \mbox{ for } k \in \{ 3, 12 \}, \\
p_{\xi}(21n+k) & \equiv 0 \pmod{4}, \mbox{ for } k \in \{ 3, 6, 12 \}, \\
p_{\xi}(33n+k) & \equiv 0 \pmod{4}, \mbox{ for } k \in \{ 6, 18, 21, 24, 30 \}.
\end{align*}

\begin{theorem} 
\label{thm:3n1_mod4}
For all $n \geq 0$, we have %$p_{\xi}(3n+1) \not\equiv 0 \pmod{4}$ if and only if $3n+1$ is a square.
$$p_{\xi}(3n+1) \equiv  
\begin{cases}
2 \pmod{4}, & \mbox{if $3n+1$ is a square}, \\
0 \pmod{4}, & \mbox{otherwise}.
\end{cases}$$
\end{theorem}

\begin{proof} From Theorem \ref{T0},
\begin{align}
\sum_{n=0}^{\infty} p_{\xi}(3n+1)q^n & = 2\frac{f_{3}f_{6}}{f_{1}}.
\end{align}
So we only need to consider the parity of 
$$\frac{f_{3}f_{6}}{f_{1}}.$$
Note that
\begin{align*}
\frac{f_{3}f_{6}}{f_{1}} \equiv \frac{f_{3}^3}{f_{1}} = \sum_{n=0}^{\infty} a_3(n)q^n \pmod{2},
\end{align*}
where $a_3(n)$ is the number of 3-core partitions of $n$ (see \cite[Theorem 1]{H-S}). Thanks to \cite[Theorem 7]{daSilva}, we know that
$$ a_{3}(n) \equiv \begin{cases}
1 \pmod{2}, & \text{if $3n+1$ is a square,}\\
0 \pmod{2}, & \text{otherwise.}
\end{cases}$$
This completes the proof.
\end{proof}

Theorem \ref{thm:3n1_mod4} yields an infinite family of congruences modulo $4$.

\begin{corollary} For all primes $p > 3$ and all $n\geq 0,$ we have
	$$p_{\xi}(3(pn+r)+1) \equiv 0 \pmod{4},$$
	if $3r+1$ is a quadratic nonresidue modulo $p$.
\end{corollary}

\begin{proof}
If $3(pn+r)+1 = k^2$, then $3r+1 \equiv k^2 \pmod{p}$, which would be a contradiction with $3r+1$ being a quadratic nonresidue modulo $p$.
\end{proof}	

For example, the following congruences hold for all $n\geq 0:$  
\begin{align*}
p_{\xi}(15n+k) & \equiv 0 \pmod{4}, \mbox{ for } k \in \{ 7, 13 \}, \\
p_{\xi}(21n+k) & \equiv 0 \pmod{4}, \mbox{ for } k \in \{ 10, 13, 19 \}, \\
p_{\xi}(33n+k) & \equiv 0 \pmod{4}, \mbox{ for } k \in \{ 7, 10, 13, 19, 28 \}.
\end{align*}

We next turn our attention to the arithmetic progression $4n+2$ to yield an additional infinite family of congruences.  

\begin{theorem} 
	\label{thm:4n2_mod4}
	For all $n \geq 0$, we have
	$$p_{\xi}(4n+2) \equiv  
	\begin{cases}
	(-1)^k \pmod{4}, & \mbox{if $n=6k(3k-1)$}, \\
	0 \pmod{4}, & \mbox{otherwise}.
	\end{cases}$$
\end{theorem}

\begin{proof}
	From \eqref{xi_4n+2}, we obtain
	\begin{align}
	\sum_{n=0}^{\infty} p_{\xi}(4n+2)q^n & \equiv \frac{f_{6}^7}{f_{3}^2f_{12}^2} +\frac{f_{12}^2}{f_{3}^2f_{6}} \equiv 2\frac{f_6^3}{f_3^2} \equiv 2f_6^2 \equiv 2f_{12} \pmod{4}.
	\label{Eq44}
	\end{align}
	Using Euler's identity (see \cite[Eq. (1.6.1)]{H})
	$$f_1 = \sum_{n=-\infty}^{\infty} (-1)^n q^{n(3n-1)/2},$$
	we obtain
	\begin{align*}
	\sum_{n=0}^{\infty} p_{\xi}(4n+2)q^n & \equiv \sum_{n=-\infty}^{\infty} (-1)^n q^{6n(3n-1)} \pmod{4},
	\end{align*}
	which concludes the proof.
\end{proof}

Theorem \ref{thm:4n2_mod4} yields an infinite family of congruences modulo $4$.

\begin{corollary}
	\label{thm:4n2_mod 4}
	Let $p > 3$ be a prime and $r$ an integer such that $2r+1$ is a quadratic nonresidue modulo $p$.   Then, for all $n \geq 0$,
	\begin{equation*}
	p_{\xi}(4(pn +r)+2) \equiv 0 \pmod{4}.
	\end{equation*}
\end{corollary}

\begin{proof}
If $pn+r = 6k(3k-1)$, then $r \equiv 18k^2-6k \pmod{p}$. Thus, $2r+1 \equiv (6k-1)^2 \pmod{p}$, which contradicts the fact that $2r+1$ is a quadratic nonresidue modulo $p$.
\end{proof}

Thanks to Corollary \ref{thm:4n2_mod 4}, the following example congruences hold for all $n\geq 0:$  
\begin{align*}
p_{\xi}(20n + j) &\equiv  0 \pmod{4}, \textrm{\ \ for\ \ } j \in \{ 6, 14 \}, \\
p_{\xi}(28n + j) &\equiv  0 \pmod{4}, \textrm{\ \ for\ \ } j \in \{ 6, 10, 26 \}, \\
p_{\xi}(44n + j) &\equiv  0 \pmod{4}, \textrm{\ \ for\ \ } j \in \{ 14, 26, 34, 38, 42 \}, \\
p_{\xi}(52n + j) &\equiv  0 \pmod{4}, \textrm{\ \ for\ \ } j \in \{ 10, 14, 22, 30, 38, 42 \}.
\end{align*}

We now provide a mod 8 characterization for $p_{\xi}(3n).$  
\begin{theorem} For all $n \geq 0$, we have %$p_{\xi}(3n) \equiv 0 \pmod{8}$ if and only if $n$ is not a square, twice a square, three times a square, nor six times a square.
$$p_{\xi}(3n) \equiv  
\begin{cases}
1 \pmod{8}, & \mbox{if $n=0$} ,\\
6(-1)^k \pmod{8}, & \mbox{if $n=k^2$}, \\
4 \pmod{8}, & \mbox{if $n = 2k^2$, $n = 3k^2$, or $n = 6k^2$}, \\
0 \pmod{8}, & \mbox{otherwise}.
\end{cases}$$
	\label{T4}
\end{theorem}

\begin{proof}
By \eqref{xi_3n}, using \eqref{Rama1} and \eqref{Eq18}, we have
\begin{align*}
\sum_{n=0}^{\infty} p_{\xi}(3n)q^n & = \frac{f_1^6f_{2}f_{3}^4}{f_{1}^8f_{6}^2} \equiv \left( \frac{f_1^2}{f_{2}} \right)^3 \left( \frac{f_3^2}{f_{6}} \right)^2 \equiv \phi(-q)^3 \phi(-q^3)^2 \\
& \equiv \left( \sum_{n=-\infty}^{\infty} (-1)^n q^{n^2} \right)^3 \left( \sum_{n=-\infty}^{\infty} (-1)^n q^{3n^2} \right)^2 \\
& \equiv \left(1+ 2\sum_{n=1}^{\infty} (-1)^n q^{n^2} \right)^3 \left(1+ 2\sum_{n=1}^{\infty} (-1)^n q^{3n^2} \right)^2 \pmod{8}
\end{align*}
which yields
\begin{align*}
\sum_{n=0}^{\infty} p_{\xi}(3n)q^n & \equiv 1 + 6\sum_{n=1}^{\infty} (-1)^n q^{n^2} + 4 \left( \sum_{n=1}^{\infty} (-1)^n q^{n^2} \right)^2 \\ 
& \ \ \ \ + 4\sum_{n=1}^{\infty} (-1)^n q^{3n^2} + 4 \left( \sum_{n=1}^{\infty} (-1)^n q^{3n^2} \right)^2 \pmod{8}.
\end{align*}

Since 
$$\left( \sum_{n=1}^{\infty} (-1)^n q^{n^2} \right)^2 \equiv \sum_{n=-\infty}^{\infty} q^{2n^2} \pmod{2},$$
we have
$$\left( \sum_{n=1}^{\infty} (-1)^n q^{3n^2} \right)^2 \equiv \sum_{n=-\infty}^{\infty} q^{6n^2} \pmod{2}.$$
Therefore
\begin{align*}
\sum_{n=0}^{\infty} p_{\xi}(3n)q^n & \equiv 1 + 6\sum_{n=1}^{\infty} (-1)^n q^{n^2} + 4\sum_{n=1}^{\infty} q^{2n^2} \\ 
& \ \ \ \ + 4\sum_{n=1}^{\infty} (-1)^n q^{3n^2} + 4 \sum_{n=1}^{\infty} q^{6n^2} \pmod{8},
\end{align*}
which completes the proof.
\end{proof}

As with the prior results, Theorem \ref{T4} provides an effective way to yield an infinite family of congruences modulo $8$.

\begin{corollary} Let $p$ be a prime such that $p \equiv \pm 1 \pmod{24}$. Then, 
$$p_{\xi}(3(pn+r)) \equiv 0 \pmod{8},$$
if $r$ is a quadratic nonresidue modulo $p$.
\end{corollary}

\begin{proof}
Since $p \equiv \pm 1 \pmod{8}$ and $p \equiv \pm 1 \pmod{12}$, it follows that $2$ and $3$ are quadratic 
% nonresidues 
residues modulo $p$. Thus, $r, 2r, 3r$, and $6r$ are quadratic 
nonresidues modulo $p$. Indeed, according to the properties of Legendre's symbol, for $j\in \{1,2,3,6\}$, we have
$$\left( \frac{jr}{p}\right) = \left( \frac{j}{p}\right)\left( \frac{r}{p}\right) = \left( \frac{r}{p}\right) = -1.$$
It follows that we cannot have 
$3(pn+r) = jk^2,$ for some $k \in \mathbb{N}$ and $j \in \{ 1, 2, 3, 6 \}$. In fact, 
$3(pn+r) = jk^2$ would imply 
$3(pn+r) \equiv 3r \equiv jk^2 \pmod{p}$. However, for $j = 1, 2, 3, 6$, this would imply that $3r$, $6r$, $r$, or $2r$, respectively, is a quadratic residue modulo $p$, which would be a contradiction since $2, 3$, and $6$ are quadratic residues modulo $p$. The result follows from Theorem \ref{T4}.
\end{proof}%}

As an example, we note that, for $p=23$ and all $n\geq 0,$ we have
\begin{align*}
p_{\xi}(69n+k) & \equiv 0 \pmod{8}, \mbox{ for } k \in \{ 15, 21, 30, 33, 42, 45, 51, 57, 60, 63, 66 \}.
\end{align*}

\begin{theorem} For all $n \geq 0$, we have 
	$$p_{\xi}(12n+4) \equiv p_{\xi}(3n+1) \pmod{8}.$$
\label{T7}
\end{theorem}
	
\begin{proof}
	Initially we use \eqref{Eq8} to extract the odd part on both sides of \eqref{xi_3n+1}. The resulting identity is
	\begin{align}
	\sum_{n=0}^{\infty} p_{\xi}(6n+4)q^n = 2\frac{f_{3}^2f_4^2f_{24}}{f_{1}^2f_8f_{12}}.
	\label{Eq46}
	\end{align}
	Using \eqref{Eq9} in \eqref{Eq46}, we obtain 
	\begin{align*}
	\sum_{n=0}^{\infty} p_{\xi}(12n+4)q^n = 2\frac{f_{2}^6f_3f_{6}}{f_{1}^5f_4^2} \equiv 2\frac{f_1^3f_{2}^6f_3f_{6}}{f_{1}^8f_4^2} \equiv 2\frac{f_3f_{6}}{f_{1}} \pmod{8}.
	\end{align*}
	The result follows using \eqref{xi_3n+1}.
\end{proof}
	
	Now we present complete characterizations of $p_{\xi}(48n+4)$ and $p_{\xi}(12n+1)$ modulo 8.
	
	\begin{theorem} For all $n \geq 0$, we have 
		$$p_{\xi}(48n+4) \equiv p_{\xi}(12n+1) \equiv  
		\begin{cases}
		2(-1)^k \pmod{8}, & \mbox{if $n=k(3k-1)$}, \\
		0 \pmod{8}, & \mbox{otherwise}.
		\end{cases}$$
	\label{T8}
	\end{theorem}
	
	\begin{proof}
		The first congruence follows directly from Theorem \ref{T7}. Replacing \eqref{Eq8} in \eqref{xi_3n+1}, we obtain
		\begin{align*}
		\sum_{n=0}^{\infty} p_{\xi}(3n+1)q^n = 2\displaystyle\frac{f_{4}f_{6}^2f_{16}f_{24}^{2}}{f_{2}^{2}f_{8}f_{12}f_{48}} + 2q\displaystyle\frac{f_{6}^2f_{8}^{2}f_{48}}{f_{2}^{2}f_{16}f_{24}}.
		\end{align*}
		Extracting the terms of the form $q^{2n}$, we have
		\begin{align*}
		\sum_{n=0}^{\infty} p_{\xi}(6n+1)q^{2n} = 2\displaystyle\frac{f_{4}f_{6}^2f_{16}f_{24}^{2}}{f_{2}^{2}f_{8}f_{12}f_{48}},
		\end{align*}
		which, after replacing $q^2$ by $q$, yields
		\begin{align}
		\sum_{n=0}^{\infty} p_{\xi}(6n+1)q^{n} = 2\displaystyle\frac{f_{2}f_{3}^2f_{8}f_{12}^{2}}{f_{1}^{2}f_{4}f_{6}f_{24}}.
		\label{Eq28}
		\end{align}
		Now we use \eqref{Eq9} to obtain
		\begin{align*}
		\sum_{n=0}^{\infty} p_{\xi}(12n+1)q^{n} & = 2\displaystyle\frac{f_{2}^3f_{6}^4}{f_{1}^4f_{12}^2} &  \\
		& \equiv 2f_2  \equiv 2 \sum_{n=-\infty}^{\infty} (-1)^n q^{n(3n-1)} \pmod{8}, & (\mbox{by \eqref{Jacobi}})
		\end{align*}
		which completes the proof.
	\end{proof}

Theorem \ref{T8} also provides an effective way to yield an infinite family of congruences modulo $8$.
	
		\begin{corollary} For all primes $p > 3$ and all $n\geq 0$, we have
		$$p_{\xi}(48(pn+r)+4) \equiv p_{\xi}(12(pn+r)+1) \equiv 0 \pmod{8},$$
		if $12r+1$ is a quadratic nonresidue modulo $p$.
	\end{corollary}	

\begin{proof}
Let $p>3$ be a prime and $12r+1$ a quadratic nonresidue modulo $p$. If $pn+r = k(3k-1)$, then $r \equiv 3k^2-k \pmod{p}$, which implies that $12r+1 \equiv (6k-1)^2 \pmod{p}$, a contradiction. The result follows from Theorem \ref{T8}.
\end{proof}

\section{Additional congruences}
\label{sec:addl_congs}

In this section, we prove several additional Ramanujan--like congruences that are not included in the results of the previous section.

\begin{theorem} For all $n \geq 0$, we have
\begin{align}
p_{\xi}(24n+19) & \equiv 0 \pmod{3}, \label{Eq12} \\ 
p_{\xi}(27n+18) & \equiv 0 \pmod{3}, \textrm{\ \ and} \label{Eq13} \\
p_{\xi}(72n+51) & \equiv 0 \pmod{3}. \label{Eq14}
\end{align}
\label{T5}
\end{theorem}
\begin{proof} Using \eqref{Eq9} we can now 2-dissect \eqref{Eq28} to obtain
\begin{align*}
\sum_{n=0}^{\infty} p_{\xi}(6n+1)q^{n} = 2\displaystyle\frac{f_{4}^3f_{12}^4}{f_{1}^{2}f_{2}^4f_{24}^2} + 4q\displaystyle\frac{f_{6}f_{8}^2f_{12}}{f_{2}^3},
\end{align*}
from which we have
\begin{align*}
\sum_{n=0}^{\infty} p_{\xi}(12n+7)q^{2n+1} = 4q\displaystyle\frac{f_{6}f_{8}^2f_{12}}{f_{2}^3}.
\end{align*}
Now, dividing both sides of the above expression by $q$ and replacing $q^2$ by $q$, we obtain
\begin{align}
\sum_{n=0}^{\infty} p_{\xi}(12n+7)q^{n} = 4\displaystyle\frac{f_{3}f_{4}^2f_{6}}{f_{1}^3}.
\label{Eq10}
\end{align}
Using \eqref{Eq11} we rewrite \eqref{Eq10} as
\begin{align*}
\sum_{n=0}^{\infty} p_{\xi}(12n+7)q^{n} = 4\displaystyle\frac{f_{4}^8f_{6}^4}{f_{2}^{9}f_{12}^2} + 12q\displaystyle\frac{f_{4}^4f_{6}^2f_{12}^2}{f_{2}^{7}}.
\end{align*}
Taking the odd parts on both sides of the last equation, we are left with
\begin{align*}
\sum_{n=0}^{\infty} p_{\xi}(24n+19)q^{n} = 12\displaystyle\frac{f_{2}^4f_{3}^2f_{6}^2}{f_{1}^{7}},
\end{align*}
which proves \eqref{Eq12}.

In order to prove \eqref{Eq13}, we use \eqref{3_dissec_1/phi} to extract the terms of the form $q^{3n}$ of \eqref{xi_3n}. The resulting identity is
\begin{align*}
\sum_{n=0}^{\infty} p_{\xi}(9n)q^{3n} & = \frac{f_{6}^2f_{9}^6}{f_{3}^4f_{18}^3},
\end{align*}
which, after replacing $q^3$ by $q$ and using \eqref{Rama2}, yields
\begin{align*}
\sum_{n=0}^{\infty} p_{\xi}(9n)q^{n} & = \frac{f_{2}^2f_{3}^6}{f_{1}^4f_{6}^3} \equiv \frac{f_{2}^2f_{3}^5}{f_{1}f_{6}^3}  = \psi(q)\frac{f_{3}^5}{f_{6}^3} \pmod{3}.
\end{align*}
By \eqref{Rama2}, we have
\begin{align*}
\sum_{n=0}^{\infty} p_{\xi}(9n)q^{n} & \equiv \frac{f_{3}^5}{f_{6}^3} \sum_{n=0}^{\infty} q^{n(n+1)/2} \pmod{3}.
\end{align*}
Since $n(n+1)/2 \not\equiv 2 \pmod{3}$ for all $n\geq 0,$  all terms of the form $q^{3n+2}$ in the last expression have coefficients congruent to $0 \pmod{3}$, which proves \eqref{Eq13}.

We now prove \eqref{Eq14}. Replacing \eqref{3_dissec_1/phi} in \eqref{xi_3n} and extracting the terms of the form $q^{3n+2}$, we obtain
\begin{align}
\sum_{n=0}^{\infty} p_{\xi}(9n+6)q^{3n+2} = 4q^2\displaystyle\frac{f_{18}^3}{f_{3}^2}.
\label{Eq15}
\end{align}
Dividing both sides of \eqref{Eq15} by $q^2$ and replacing $q^3$ by $q$, we have
\begin{align}
\sum_{n=0}^{\infty} p_{\xi}(9n+6)q^{n} = 4\displaystyle\frac{f_{6}^3}{f_{1}^2}.
\label{Eq16}
\end{align}
Now we use \eqref{Eq17} to extract the odd part of \eqref{Eq16} and obtain
\begin{align*}
\sum_{n=0}^{\infty} p_{\xi}(18n+15)q^{n} = 8\displaystyle\frac{f_{2}^2f_3^3f_8^2}{f_{1}^5f_4}.
\end{align*}
Since $f_1^3 \equiv f_3 \pmod{3}$, we have
\begin{align*}
\sum_{n=0}^{\infty} p_{\xi}(18n+15)q^{n} \equiv 2\displaystyle\frac{f_{2}^2f_3^2f_8^2}{f_{1}^2f_4} \pmod{3}.
\end{align*}
Using \eqref{Eq9} we obtain
\begin{align*}
\sum_{n=0}^{\infty} p_{\xi}(36n+15)q^{n} \equiv 2\displaystyle\frac{f_{2}^3f_3f_4f_6^2}{f_{1}^3f_{12}} \pmod{3}.
\end{align*}
Since the odd part of \eqref{Eq11} is divisible by $3$, then the coefficients of the terms of the form $q^{2n+1}$ in $\sum_{n=0}^{\infty} p_{\xi}(36n+15)q^{n}$ are congruent to $0$ modulo $3$. This completes the proof of \eqref{Eq14}.	
\end{proof}

\begin{theorem} For all $n \geq 0$, we have
\begin{align}
p_{\xi}(8n+6) & \equiv 0 \pmod{4}, \label{Eq41} \\
%p_{\xi}(12n+10) & \equiv 0 \pmod{4}, \label{Eq41.1} \\
p_{\xi}(16n+10) & \equiv 0 \pmod{4}. \label{Eq41.2}
\end{align}
\end{theorem}

\begin{proof}
Congruence \eqref{Eq41} follows directly by extracting the odd part of \eqref{Eq44}.  Extracting  the even part of \eqref{Eq44} yields
\begin{align*}
\sum_{n=0}^{\infty} p_{\xi}(8n+2)q^n &  \equiv 2f_{6} \pmod{4},
\end{align*}
from which \eqref{Eq41.2} follows. 
\end{proof}

We now prove a pair of unexpected congruences modulo 5 satisfied by $p_{\xi}(n).$  		
\begin{theorem} For all $n \geq 0$, we have
\begin{align}
p_{\xi}(45n+33) & \equiv 0 \pmod{5}, \label{Eq27} \\
p_{\xi}(45n+41) & \equiv 0 \pmod{5}. \label{Eq32}
\end{align}
\end{theorem}

\begin{proof}
By \eqref{Eq16}, we have
\begin{align*}
\sum_{n=0}^{\infty} p_{\xi}(9n+6)q^{n} & = 4\frac{f_{6}^3}{f_{1}^2} = 4\frac{f_1^3f_{6}^3}{f_{1}^5} \equiv 4\frac{f_1^3f_{6}^3}{f_{5}} \pmod{5}.
\end{align*}
Thanks to Jacobi's identity \eqref{Jacobi} we know
$$f_1^3f_6^3 = \sum_{j,k=0}^{\infty} (-1)^{j+k} (2j+1)(2k+1) q^{3j(j+1)+k(k+1)/2}.$$
Note that, for all integers $j$ and $k$, $3j(j+1)$ and $k(k+1)/2$ are congruent to either $0$, $1$ or $3$ modulo $5$. The only way to obtain $3j(j+1)+k(k+1)/2 = 5n+3$ is the following:
\begin{itemize}
	\item $3j(j+1) \equiv 0 \pmod{5}$ and $k(k+1)/2 \equiv 3 \pmod{5}$, or
	\item $3j(j+1) \equiv 3 \pmod{5}$ and $k(k+1)/2 \equiv 0 \pmod{5}$.
\end{itemize}
Thus, $j \equiv 2 \pmod{5}$ or $k \equiv 2 \pmod{5}$ in all possible cases, and this means 
$$(2j+1)(2k+1) \equiv 0 \pmod{5}.$$
Therefore, for all $n \geq 0$, $p_{\xi}(45n+33) = p_{\xi}(9(5n+3)+6) \equiv 0 \pmod{5}$, which is \eqref{Eq27}.

In order to complete the proof of \eqref{Eq32}, we want to see when $$3j(j+1)+k(k+1)/2 = 5n+4.$$ Four possible cases arise:
\begin{itemize}
	\item $k \equiv 1 \pmod{5}$ and $j \equiv 2 \pmod{5}$,
	\item $k \equiv 3 \pmod{5}$ and $j \equiv 2 \pmod{5}$,
	\item $j \equiv 1 \pmod{5}$ and $k \equiv 2 \pmod{5}$, \textrm{\ \ or}
	\item $j \equiv 3 \pmod{5}$ and $k \equiv 2 \pmod{5}$.
\end{itemize}
In all four cases above, either $j \equiv 2 \pmod{5}$ or $k \equiv 2 \pmod{5}$. So $$(2j+1)(2k+1) \equiv 0 \pmod{5}$$ in all these cases. Therefore, $$p_{\xi}(45n+42) = p_{\xi}(9(5n+4)+6) \equiv 0 \pmod{5},$$ which completes the proof of \eqref{Eq32}.
\end{proof}
		
Next, we prove three congruences modulo 8 which are not covered by the above results.  
\begin{theorem} For all $n \geq 0$, we have
\begin{align}
p_{\xi}(16n+14) & \equiv 0 \pmod{8}, \label{Eq42} \\
p_{\xi}(24n+13) & \equiv 0 \pmod{8}, \label{Eq25} \\
p_{\xi}(24n+22) & \equiv 0 \pmod{8}.  \label{Eq26}
\end{align}
\end{theorem}

\begin{proof} Initially we prove \eqref{Eq42}. From \eqref{xi_4n+2} and \eqref{Rama1} we have
\begin{align*}
\sum_{n=0}^{\infty} p_{\xi}(4n+2)q^n & \equiv \frac{f_{3}^2f_6^5}{f_{12}^2} +\frac{f_{12}^2}{f_{3}^2f_{6}} \phi(q)^2 \pmod{8}.
\end{align*}
Now we can use \eqref{Eq17}, \eqref{Eq31}, and \eqref{phi^2-2-dissec} to extract the terms involving $q^{2n+1}$ from both sides of the previous congruence:
\begin{align*}
\sum_{n=0}^{\infty} p_{\xi}(8n+6)q^{2n+1} & \equiv -2q^3\frac{f_{6}^6f_{48}^2}{f_{12}^2f_{24}} +2q^3\frac{f_4^{10}f_{12}^4f_{48}^2}{f_{2}^4f_{6}^6f_8^4f_{24}} +4q\frac{f_{8}^4f_{12}^2f_{24}^5}{f_{4}^2f_6^6f_{48}^2} \pmod{8}.
\end{align*}
After dividing both sides by $q$ and then replacing $q^2$ by $q$, we are left with
\begin{align*}
\sum_{n=0}^{\infty} p_{\xi}(8n+6)q^{n} & \equiv -2q\frac{f_{3}^6f_{24}^2}{f_{6}^2f_{12}} +2q\frac{f_2^{10}f_{6}^4f_{24}^2}{f_{1}^4f_{3}^6f_4^4f_{12}} +4\frac{f_{4}^4f_{6}^2f_{12}^5}{f_{2}^2f_3^6f_{24}^2} \\ 
& \equiv -2q\frac{f_{3}^6f_{24}^2}{f_{6}^2f_{12}} +2q\frac{f_3^{6}f_{6}^4f_{24}^2}{f_{3}^{12}f_{12}} +4\frac{f_{4}^4f_{12}^5}{f_{2}^2f_6f_{24}^2} \\ & \equiv 4\frac{f_{4}^3f_{12}}{f_6} \pmod{8},
\end{align*}
whose odd part is congruent to 0 modulo $8$, which implies \eqref{Eq42}.
	
In order to prove \eqref{Eq25}, we use \eqref{Eq9} to obtain the even part of identity \eqref{Eq28}, which is
\begin{align*}
\sum_{n=0}^{\infty} p_{\xi}(12n+1)q^{n} = 2\frac{f_{2}^3f_6^4}{f_1^4f_{12}^2}.
\end{align*}
Now, employing \eqref{Eq29}, we obtain the odd part of the last identity, which is
\begin{align*}
\sum_{n=0}^{\infty} p_{\xi}(24n+13)q^{n} = 8\frac{f_{2}^2f_3^4f_4^4}{f_1^7f_{6}^2},
\end{align*}
which implies \eqref{Eq25}.

Now we prove \eqref{Eq26}. We employ \eqref{Eq9} in \eqref{Eq46} to obtain 
\begin{align}
\sum_{n=0}^{\infty} p_{\xi}(12n+10)q^{n} = 4\frac{f_{2}^3f_3^2f_{12}^2}{f_{1}^4f_6},
\label{Eq30}
\end{align}
By \eqref{Eq31} and \eqref{Eq29}, we rewrite \eqref{Eq30} in the form
\begin{align*}
\sum_{n=0}^{\infty} p_{\xi}(12n+10)q^{n} = 4\frac{f_{2}^3f_{12}^2}{f_{6}^2} \left( \frac{f_4^{14}}{f_2^{14}f_8^4} +4q \frac{f_4^2f_8^4}{f_2^{10}} \right) \left( \frac{f_6f_{24}^5}{f_{12}^2f_{48}^2} -2q^3 \frac{f_6f_{48}^2}{f_24} \right),
\end{align*}
from which we obtain
\begin{align*}
\sum_{n=0}^{\infty} p_{\xi}(24n+22)q^{2n+1} = 4\frac{f_{2}^3f_{12}^2}{f_{6}^2} \left( -2q^3\frac{f_4^{14}f_6f_{48}^2}{f_2^{14}f_8^4f_{24}} +4q \frac{f_4^2f_6f_8^4f_{24}^5}{f_2^{10}f_{12}^2f_{48}^2} \right).
\end{align*}
Dividing both sides by $q$ and replacing $q^2$ by $q$, we are left with
\begin{align*}
\sum_{n=0}^{\infty} p_{\xi}(24n+22)q^{n} =  -8q\frac{f_2^{14}f_6^2f_{24}^2}{f_1^{11}f_3f_4^4f_{12}} +16 \frac{f_2^2f_4^4f_{12}^5}{f_1^{7}f_3f_{24}^2},
\end{align*}
which implies \eqref{Eq26}.
\end{proof}		
		
We close this section by proving a congruence modulo 9.  		
\begin{theorem} 
\label{thm:96_76}
For all $n \geq 0$, we have
\begin{align}
p_{\xi}(96n+76) & \equiv 0 \pmod{9}. \label{Eq43}
\end{align}
\end{theorem}		

\begin{proof}
We use \eqref{3_dissec_psi} to extract the terms of the form $q^{3n+1}$ from \eqref{xi_4n}. The resulting identity is
\begin{align*}
\sum_{n=0}^{\infty} p_{\xi}(12n+4)q^{3n+1} = 2q\frac{f_{6}^6f_9f_{18}}{f_3^5f_{12}^2},
\end{align*}
which, after dividing by $q$ and replacing $q^3$ by $q$, yields
\begin{align*}
\sum_{n=0}^{\infty} p_{\xi}(12n+4)q^{n} = 2\frac{f_{2}^6f_3f_{6}}{f_1^5f_{4}^2} = 2\frac{f_{2}^6f_{6}}{f_{4}^2} \frac{f_3}{f_1} \frac{1}{f_1^4}.
\end{align*}
Using \eqref{Eq29} and \eqref{Eq8}, we extract the even part on both sides of the above identity to obtain 
\begin{align*}
\sum_{n=0}^{\infty} p_{\xi}(24n+4)q^{n} & = 2\frac{f_{2}^{13}f_3^2f_{8}f_{12}^2}{f_1^{10}f_{4}^5f_6f_{24}} + 8q\frac{f_{3}^2f_{4}^6f_{24}}{f_1^6f_8f_{12}}  \\
& \equiv 2\frac{f_{2}^{13}f_{8}f_{12}^2}{f_{4}^5f_6f_{24}}\frac{1}{f_1f_3} + 8q\frac{f_{4}^6f_{24}}{f_8f_{12}} \frac{f_1^3}{f_3} \pmod{9}.
\end{align*}
Now we employ \eqref{Eq38} and \eqref{Eq45} to extract the odd part on both sides of the last congruence:
\begin{align*}
\sum_{n=0}^{\infty} p_{\xi}(48n+28)q^{n} & \equiv 2\frac{f_{1}^{9}f_6f_{12}}{f_{3}^3f_4} + 8\frac{f_{2}^9f_{12}}{f_4f_{6}^2} \equiv \frac{f_6f_{12}}{f_4} \pmod{9},
\end{align*}
which implies \eqref{Eq43}.
\end{proof}

\section{Concluding remarks}

Computational evidence indicates that $p_{\xi}(n)$ satisfies many other congruences.  The interested reader may wish to consider the following two conjectures.  

\begin{conjecture}
\label{Conj1}
$$\sum_{n=0}^\infty p_{\xi}(8n+3)q^n \equiv 2\sum_{n=0}^\infty q^{3n(n+1)/2} \pmod{3}$$
\end{conjecture}

\begin{conjecture}
\label{Conj2}
$$\sum_{n=0}^\infty p_{\xi}(32n+12)q^n \equiv 6\sum_{n=0}^\infty q^{3n(n+1)/2} \pmod{9}$$
\end{conjecture}

Clearly, once proven, Conjectures \ref{Conj1} and \ref{Conj2} would immediately lead to infinite families of Ramanujan--like congruences.  Morever, Conjecture \ref{Conj2} would immediately imply Theorem \ref{thm:96_76} since $96n+76 = 32(3n+2)+12$ while the right--hand side of Conjecture \ref{Conj2} is clearly a function of $q^3.$  The same argument would imply that, for all $n\geq 0,$ 
$$p_{\xi}(96n+44)  \equiv 0 \pmod{9}$$
since $96n+44 = 32(3n+1)+12.$

\section*{Acknowledgment}

The first author was supported by S\~ao Paulo Research Foundation (FAPESP) (grant no. 2019/14796-8).

\

\noindent Universidade Federal de S\~ao Paulo, Av. Cesare M. G. Lattes, 1201, S\~ao Jos\'e dos Campos, SP, 12247--014, Brazil. \\
E-mail address: silva.robson@unifesp.br

\

\noindent Department of Mathematics and Statistics, University of Minnesota Duluth, Duluth, MN  55812, USA. \\
E-mail address: jsellers@d.umn.edu

\end{document}